%% file: main.tex
\documentclass{amsart}
\usepackage[utf8]{inputenc}
\usepackage{amsthm,mathtools,amssymb,physics}
\usepackage{overpic,hyperref,cleveref}
\input{command}
\begin{document}
\input{p/00info}
\input{p/01abst}
\input{p/02intro}
\input{p/03clas}
\input{p/04guide}
\input{p/10nota}
\input{p/11ass}
\input{p/12assalt}
\input{p/12resu}
\input{p/20prelCMest}
\input{p/21prelCM}
\input{p/22prelPot}
\input{p/30proofldc}
\input{p/31proof}
\input{p/32proofH5}
\input{p/40inv}
\input{p/41invpotpro}
\input{p/42invpotex}
\input{p/51remain}
\input{p/99-refs}
\end{document}

%% file: command.tex
\input{com/theo}
\input{com/mathop}
\input{com/set}
\input{com/index}
\newcommand{\ep}{\varepsilon}
\newcommand{\la}{\Delta}
\newcommand{\Lip}[1]{\abs{#1}_{C^{0,1}}}
\newcommand{\pd}{\partial}
\newcommand{\duS}[2]{\ev{#1,#2}_{H^{-1}(\Omega)}}
\newcommand{\loc}{{\rm loc}}

%% file: com/theo.tex
\newtheorem{defi}{Definition}
\crefname{defi}{Definition}{}
\newtheorem{lem}{Lemma}
\crefname{lem}{Lemma}{Lemmas}
\newtheorem{theo}{Theorem}
\crefname{theo}{Theorem}{Theorems}
\newtheorem{prop}{Proposition}
\crefname{prop}{Proposition}{Propositions}
\newtheorem{rem}{Remark}
\crefname{rem}{Remark}{}
\newtheorem{ex}{Example}
\crefname{ex}{Example}{Examples}
\newtheorem{ass}{Assumption}[section]
\crefname{ass}{Assumption}{Assumptions}

%% file: com/mathop.tex
\DeclareMathOperator{\capa}{cap}
\DeclareMathOperator{\supp}{supp}
\DeclareMathOperator{\diam}{diam}

%% file: com/set.tex
\newcommand{\D}{\mathcal{D}}

\newcommand{\R}{\mathbb{R}}
\newcommand{\Z}{\mathbb{Z}}

\newcommand{\cls}[1]{\overline{#1}}

%% file: com/index.tex
\newcommand{\ai}{a_{i\ep}}
\newcommand{\Ri}{R_{i\ep}}
\newcommand{\ri}{r_{i\ep}}
\newcommand{\Ki}{K_i^\ep}
\newcommand{\Hi}{H_{i\ep}}
\newcommand{\Ai}{A_i^\ep}
\newcommand{\phii}{\phi_i^\ep}
\newcommand{\Bi}{B(x_i^\ep,\Ri)}

%% file: p/00info.tex
\title[]{Inverse homogenization problem for the Drichlet problem for Poisson equation for $W^{-1,\infty}$ potential}
\author[H. Ishida]{Hiroto Ishida}
\address{Graduate School of Science, University of Hyogo\\Shosha, Himeji, Hyogo 671-2201, Japan}
\email{immmrfff@gmail.com}
\date{\today}
\maketitle

%% file: p/01abst.tex
\begin{abstract}
We consider Poisson problems $-\la u^\ep=f$ on perforated domains,
and characterize the limit of $u^\ep$ as the solution to $(-\la+\mu)u=f$ on domain $\Omega\subset\R^d$ with some potential $\mu\in W^{-1,\infty}(\Omega).$ It is known that $\mu$ is related to the capacity of holes when $\mu\in L^\infty(\Omega).$
In this paper, we characterize $\mu$ as the limit of the density of the capacity of holes also for many $\mu\in W^{-1,\infty}(\Omega).$ We apply the result for the inverse homogenization problem, i.e. we construct holes corresponding to the given potential $\mu\in L^d(\Omega)+L^\infty(\delta_S)$ where $\delta_S$ is a surface measure.
\end{abstract}

%% file: p/02intro.tex
\section{Introduction}
Let $\Omega\subset\R^d~(d\geq3)$ be open and bounded with $C^2$ boundary, and $f\in L^2(\Omega).$ Let $\{\Ki\}_{i\in\Lambda^\ep}$ be at most countable family of compact sets of $\R^d$ for each $\ep>0.$ We consider Poisson problems with Dirichlet condition in $\Omega\setminus\bigcup_{i\in\Lambda^\ep}\Ki$, that is,
\begin{equation}\label{poi}\tag{PDE1}
u^\ep\in H_0^1\qty(\Omega\setminus\bigcup_{i\in L^\ep}\Ki),~-\la u^\ep=f.
\end{equation}
We regard $H_0^1(\Omega\setminus\bigcup_{i\in\Lambda^\ep}\Ki)\subset H_0^1(\Omega)$ by the zero extension. The aim of this paper is to see $u^\ep\to u$ weakly in $H_0^1(\Omega)$ ("as $\ep\to 0$" is omitted in this paper) under the assumptions in \Cref{ass}, where $u$ is the solution to
\begin{equation}\label{poipot}\tag{PDE2}
u\in H_0^1(\Omega),(-\la+\mu)u=f
\end{equation}
with some $\mu\in W^{-1,\infty}(\Omega)$ which is considered as the limit of the density of the capacity of holes (see \eqref{ldc}). The assumptions in \Cref{ass} mean holes are separated each other in a proper sense.

We also consider the inverse homogenization problem. That is a problem to construct holes $\{\Ki\}$ such that $u^\ep\to u$ for a given potential $\mu.$

%% file: p/03clas.tex
\subsection{Known results}
The PDE of the form \eqref{poipot} is often used to characterize the limit $u$.
Examples for which $\mu=0$ are introduced at \cite{RT} with domains perforated by compact sets which have enough small capacities.
On the other hand, there are examples for which $\mu\neq0.$
Cases that holes are closed (compact or unbounded) sets are introduced at \cite{CM}. For example, \cite[Example 2.1.]{CM} is the case perforated by periodic balls distributed on $\R^d.$ It is generalized to the case that holes are concentrated balls in \cite{Con}. 

Relation between the capacity (resp. the relative capacity) of holes and $\mu$ is suggested in \cite[Theorem 2.3.]{KM} (resp. \cite{Hru2}) in the case that holes are compact sets and $\mu\in L^\infty.$ On the other hand, there is an example such that $\mu\notin L^\infty.$ For example, $\mu$ is a surface measure on a hyper plane in \cite[Example 2.9.]{CM}.

Our result characterizes $\mu$ as the limit of the density of the capacity of holes (see \eqref{ldc}) even for many $\mu\in W^{-1,\infty}.$ The result can be also applied to the inverse homogenization problem, that is, a problem to construct holes corresponding to a given potential such as $\mu\in L^d+L^\infty(\delta_S)$ where $\delta_S$ is a surface measure.

Convergence rate for Dirichlet Laplacians for $\mu\in \bigcup_{p>d}L^p(\Omega)$ is studied by the author of this paper in \cite{RATE}. However, the result requires additional assumptions. In particular, we still need the main result of this paper for $\mu\in W^{-1,\infty}$ such as a surface measure.

The capacities of several sets are known such as ball in $\R^d$, ellipsoid, torus, semi-ball, union of two balls in $\R^3$ (see also \cite[Chap. II, Sec. 3. No. 14]{NSL}). It is one of the advantage to use the capacities of holes to characterize $\mu.$

%% file: p/04guide.tex
This paper is organized as follows. We introduce notation in \Cref{nota}. We state the main result in \Cref{sresu} under the assumptions in \Cref{ass}. We show the main result in \Cref{spr}.
Finally, we show the result for the inverse homogenization problem in \Cref{sinv}.

%% file: p/10nota.tex
\section{Notation}\label{nota}
We use the following notation:

$S_d$ is the surface area of the unit sphere in $\R^d.$

$\Lip{g}$ is the minimal Lipschitz constant of function $g.$

$\diam E$ is diameter of $E\subset\R^d.$

$1_E(x)=\begin{cases}
1&(x\in E)\\
0&(x\notin E)\end{cases}$ for $E\subset\R^d,~x\in\R^d.$

$B(x,r)=\{y\in\R^d\mid|x-y|<r\},~B(x,0)=\emptyset$ for $x\in\R^d,~r>0.$

$|E|$ is Lebesgue measure of Lebesgue measurable set $E\subset\R^d.$

$\delta_S$ is the surface measure of a surface $S.$

Moreover, we also use non standard notation:

We denote $\sum_{\substack{i\in\Lambda^\ep\\\Omega\cap \Ki\neq\emptyset}}$ by $\sum_i.$ Similar for $\bigcup_i$ and $\sup_i.$

%% file: p/11ass.tex
\section{Assumptions and the main result}\label{ass}
In this section, we state assumptions and the main theorem.
For convenience, "for each $0<\ep\ll1,$ and $i\in\Lambda^\ep$" is denoted by "for each $\ep,~i$".
\begin{ass}\label{assholec1}
$\pd\Ki$ is piecewise $C^1$ for each $\ep,~i.$
\end{ass}
We recall representation of the Newtonian capacity for a compact set with piecewise $C^1$ boundary (see also \cite[(2.25),(2.26)]{KM}).
\begin{defi}[capacity]
We let $\Hi$ the function such that
\begin{equation}\label{cappot}
\left\{
\begin{array}{r@{~}ll}
\Hi&=1&\mbox{on }\Ki,\\
\la\Hi&=0&\mbox{on }\R^d\setminus \Ki,\\
\Hi(x)&\to0&(|x|\to\infty)
\end{array}\right.
\end{equation}
for each $\ep,~i.$ The capacity of each $K_i^\ep$ is defined by 
\[\capa(K_i^\ep)=\norm{\grad \Hi}_{L^2(\R^d)}^2.\]
\end{defi}
We denote $\cls{B(x_i^\ep,\ai)}$ the minimal ball such that $K_i^\ep\subset\cls{B(x_i^\ep,\ai)}$ for each $\ep,~i.$
Before stating our assumptions, we remark the embedding
\begin{equation}\label{embedding}
L^\infty(V)\hookrightarrow L^d(V)\hookrightarrow W^{-1,\infty}(V)\hookrightarrow H^{-1}(V)\mbox{ for every bounded open }V
\end{equation}
follows from embedding theorem $H_0^1(V)\hookrightarrow W_0^{1,1}(V)\hookrightarrow L^{(1-1/d)^{-1}}(V)\hookrightarrow L^1(V).$
\begin{defi}
For an open set $V\subset\R^d,$ we let
\[
H^{-1}(V)^+=\{\mu\geq0\colon\mbox{Borel measure on }V\mid\mu\in H^{-1}(V)\}.
\]
We also define
\[H^{-1}_\loc(\R^d)^+=\{\mu\geq0\colon\mbox{Borel measure on }\R^d\mid\mu\in H^{-1}(V)\mbox{ for every ball }V\subset\R^d\}.\]
We define $W^{-1,\infty}(V)^+$ and $W^{-1,\infty}_\loc(\R^d)^+$ similarly.

For a sequence $\{\nu_\ep,\nu\}\subset H^{-1}_\loc(\R^d)^+,$ we say
$\nu_\ep\to\nu$ in $H^{-1}_\loc(\R^d)$ if $\norm{\nu_\ep-\nu}_{H^{-1}(V)}$
$\to0$ for every ball $V\subset\R^d.$
\end{defi}
\begin{ass}\label{asss}
There exists $\Ri$ for each $\ep,~i$ such that $\{\Bi\}_{i\in\Lambda^\ep}$ is a family of disjoint sets and there exists Lebesgue measurable \begin{equation}\label{supertile}\Ai\supset\Bi\end{equation}
for each $\ep,~i$ such that
\begin{gather}
\label{A0}\tag{A0}
\sup_{\ep<1}\sup_{i}\Ri<\infty,\\
\label{A1}\tag{A1}
\sup_{\ep<1}\sup_i\ai/\Ri<1,\\
\label{A2}\tag{A2}
\sum_i\ai^{2(d-2)}\Ri^{-d+2}\to 0,\\
\label{A3}\tag{A3}
\sup_{\ep<1}\sup_i|\Ai|\Ri^{-d}<\infty,\\
\label{A4}\tag{A4}
\sum_i\ai^{d-2}\diam A_i^\ep\to 0,\\
\label{A5}\tag{A5}
\qty{\sum_i\frac{\ai^{d-2}}{|\Ai|}1_{\Ai}}\mbox{ converges in the norm of }H^{-1}(\Omega).\\
\label{A6}\tag{A6}
\sup_{\ep<1}\sum_i\ai^{d-2}<\infty.\\
\label{ldc}\tag{ldc}
\sum_i\frac{\capa(\Ki)}{|\Ai|}1_{\Ai}\to\mu\mbox{ in }\D'(\Omega).
\end{gather}
\end{ass}

%% file: p/12assalt.tex
\begin{rem}
We have \begin{equation}\label{Jung}
\ai\leq\diam\Ki\leq2\ai
\end{equation} by Jung's theorem. Therefore, we can replace $\ai$ in \Cref{asss} with $\diam\Ki.$
\end{rem}
We have a simple alternative for \Cref{asss} as below.
\begin{prop}\label{assalt}
\Cref{asss} replaced \eqref{A1} to \eqref{A6} with
\begin{gather}\tag{a1}\label{assalt1}
\sup_i\ai/\Ri\to0,\\
\tag{a2}\label{assalt2}
\sup_{\ep<1}\sup_i(\diam\Ai)/\Ri<\infty,\\
\tag{a3}\label{assalt3}
\qty{\sum_i\frac{\ai^{d-2}}{|\Ai|}1_{\Ai}}\mbox{ converges in }H^{-1}_\loc(\R^d)
\end{gather}
imply \Cref{asss}.
\begin{proof}
\eqref{assalt1} implies \eqref{A1}. \eqref{assalt3} implies \eqref{A5}.

Let $c=\sup_i\diam\Ai/\Ri,~c'=c\sup_i\Ri$ and $V=\bigcup_{x\in\Omega}B(x,c').$ Take $g\in\D(\R^d)$ such that $g\geq 1_V.$
\eqref{A6} follows from \eqref{assalt3} and
\[\sum_i\ai^{d-2}
\leq\ev{\sum_i\frac{\ai^{d-2}}{|\Ai|}1_{\Ai},g}.\]
\eqref{A2} follows from \eqref{assalt1}, \eqref{A6} and
\[\sum_i\ai^{2(d-2)}\Ri^{-d+2}
\leq(\sup_i\ai/\Ri)^{d-2}\sum_i\ai^{d-2}.\]
\eqref{A3} follows from
$|\Ai|<|B(0,\diam\Ai)|$
$\leq |B(0,c\Ri)|\leq|B(0,c)|\Ri^d.$ \eqref{A4} follows from \eqref{A2}, $|B(0,1)|\sum_i\Ri^d=\sum_i|B(x_i^\ep,\Ri)|\leq|V|$ and
\[\sum_i\ai^{d-2}\diam\Ai
\leq c\sum_i\ai^{d-2}\Ri^{(-d+2)/2}\Ri^{d/2}
\leq c\sqrt{\sum_i\ai^{2(d-2)}\Ri^{-d+2}}\sqrt{\sum_i\Ri^d}.\]
\end{proof}
\end{prop}

%% file: p/12resu.tex
\subsection{Result}\label{sresu}
The main result is stated as below. 
\begin{theo}\label{resu}
Under \Cref{assholec1,asss}, we have $u^\ep\to u$ weakly in $H_0^1(\Omega),$ where $u^\ep$ and $u$ are the solutions to \eqref{poi} and \eqref{poipot} respectively.
\end{theo}
We have another characterization for $\mu$ as below.
\eqref{ldc2} is an expression which only relies on holes $\{\Ki\}.$
\begin{prop}\label{ldcconvtop}
Under \Cref{assholec1,asss}, we have
\begin{equation}\tag{ldc'}\label{ldc3}
\norm{\sum_i\frac{\capa(\Ai)}{|\Ai|}1_{\Ai}-\mu}_{H^{-1}(\Omega)}\to0.
\end{equation}
If $|\Ki|\neq\emptyset$ for $\ep,~i,$ we have
\begin{equation}\tag{ldc''}\label{ldc2}
\sum_i\frac{\capa(\Ki)}{|\Ki|}1_{\Ki}\to\mu\mbox{ in }\D'(\Omega).
\end{equation}
\end{prop}
The answer for inverse for homogenization problem is given as below.
\begin{theo}\label{invperi}
Let $S$ be a piece wise $C^1$ surface on $\R^d$ and $\mu\in L^d(\R^d,[0,\infty))+L^\infty(\delta_S,[0,\infty))$ (remark $\mu$ is a Borel measure on $\R^d$).
Let $\Lambda=\Lambda^\ep=2\Z^d,~\Ki=\cls{B(2\ep i,\qty(\frac{\mu(A_i^\ep)}{(d-2)S_d})^\frac{1}{d-2})}$ for each $\ep>0$ and $i\in\Lambda$ (see \Cref{figperi}). Then, the solutions $u^\ep$ and $u$ to \eqref{poipot} and \eqref{poi} satisfy $u^\ep\to u$ weakly in $H_0^1(\Omega).$
\end{theo}
\begin{figure}\centering
\begin{overpic}[width=8cm]{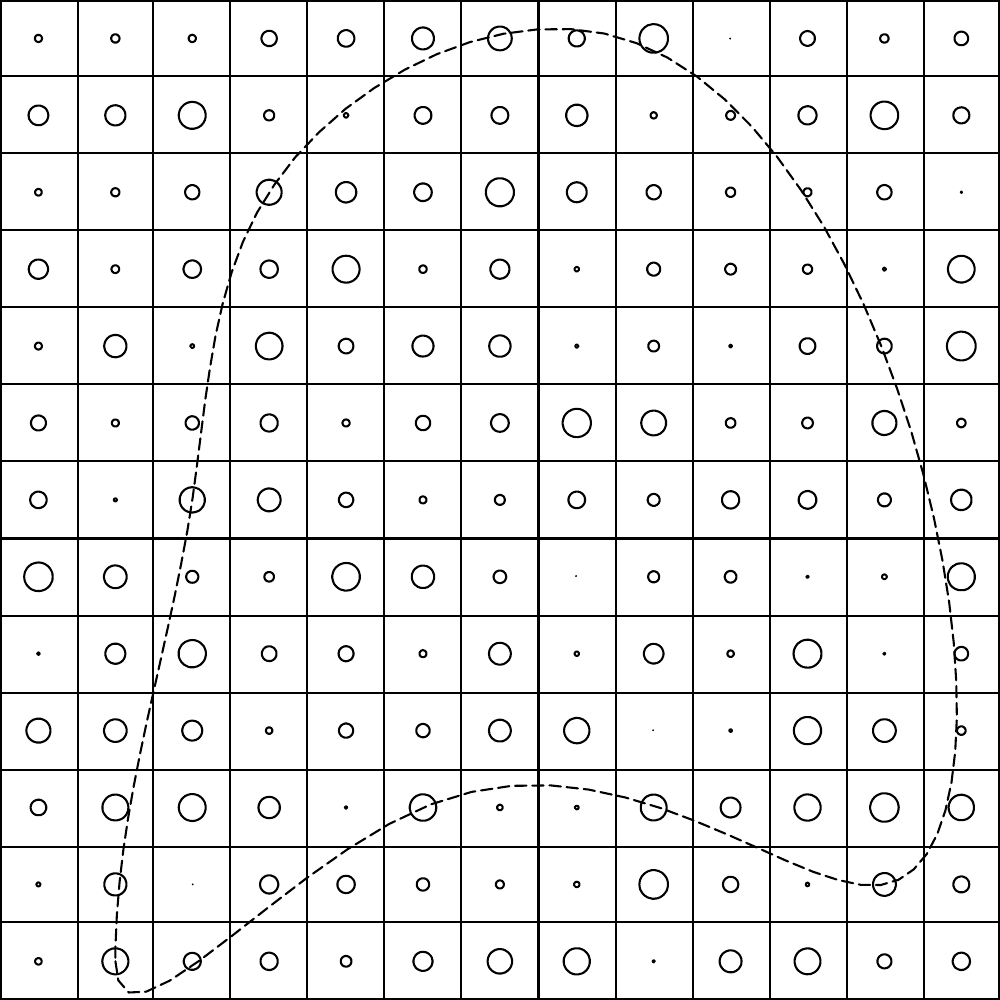}
\end{overpic}
\caption{A domain perforated by periodic balls\label{figperi}}
\end{figure}

%% file: p/20prelCMest.tex
\section{Proof}\label{spr}
In \Cref{prel}, we introduce some results needed to show \Cref{resu}.
We show \Cref{resu} in \Cref{prf}.
\subsection{Preliminaries to proof}\label{prel}
\Cref{conv2} below is based on \cite[Lemma 2.8]{CM}.
\begin{lem}\label{conv2}
Let $\{\nu_+^\ep,\nu_-^\ep\}_{\ep>0}\subset H^{-1}(\Omega)^+,~\nu\in H^{-1}(\Omega),~\nu_+^\ep-\nu_-^\ep\to\nu$ in $\D'(\Omega),$ and there exists $\{N^\ep\}_{\ep>0}\subset H^{-1}(\Omega)^+$ converging in the norm of $H^{-1}(\Omega)$, and $\nu_+^\ep+\nu_-^\ep\leq N^\ep.$ Then, we have $\norm{\nu_+^\ep-\nu_-^\ep-\nu}_{H^{-1}(\Omega)}\to 0.$
\begin{proof}
Let $t^\ep=\norm{\nu_+^\ep-\nu_-^\ep-\nu}_{H^{-1}(\Omega)},~t=\limsup_{\ep\to 0}t^\ep.$
By \cite[Lemma 2.8]{CM}\footnote{The word "strongly convergence" in \cite[Lemma 2.8]{CM} means convergence in norm.}, there exist subsequences (still denoted by $\ep$) such that $t^\ep\to t$ and $\{\nu_\pm^\ep\}$ converges in the norm of $H^{-1}(\Omega).$ The limit of $\{\nu_+^\ep-\nu_-^\ep\}\subset H^{-1}(\Omega)$ coincides with the limit $\nu$ in $\D'(\Omega)$ since $H^{-1}(\Omega)\hookrightarrow\D'(\Omega).$ Therefore, we have $t=0.$
\end{proof}
\end{lem}

%% file: p/21prelCM.tex
\begin{theo}[{\cite[Theorem 1.2]{CM}}]\label{CM}
Assume there exists a sequence
\begin{equation}\tag{H.1}\label{H1}
\{w^\ep\}\subset H^1(\Omega)
\end{equation}
satisfying
\begin{equation}\tag{H.2}\label{H2}
w^\ep=0\mbox{ on }\bigcup_iK_i^\ep\mbox{ for each }\ep>0,
\end{equation}
\begin{equation}\tag{H.3}\label{H3}
w^\ep\to 1\mbox{ weakly in }H^1(\Omega),
\end{equation}
\begin{equation}\tag{H.4}\label{H4}
\mu\in W^{-1,\infty}(\Omega),
\end{equation}
and
\begin{equation}\tag{H.5}\label{H5}
\begin{aligned}
&\duS{-\la w^\ep}{g v^\ep}\to\duS{\mu}{g v}\\
&\mbox{ if }g\in\D(\Omega),~
v^\ep=0\mbox{ on }\bigcup_iK_i^\ep\mbox{ and }v^\ep\to v \mbox{ weakly in }H^1(\Omega).
\end{aligned}
\end{equation}
Then, the solutions $u^\ep$ to \eqref{poi} converges to the solution 
$u$ to \eqref{poipot} weakly in $H_0^1(\Omega).$
\end{theo}

%% file: p/22prelPot.tex
\begin{lem}[{\cite[2.2.2 Lemma 2.4.]{KM}}]
There exists $c_0>0$ such that
\begin{equation}\label{pot1}\tag{P1}
|\pd^\alpha\Hi(x)|
\leq c_0\ai^{d-2}(|x-x_i^\ep|-\ai)^{-d+2-|\alpha|}~(|x-x_i^\ep|-\ai\geq c_0\ai).
\end{equation}
\begin{proof}
See \cite[2.2.2 Lemma 2.4.]{KM} and \eqref{Jung}.
\end{proof}
\end{lem}
We denote $\ri=\Ri-\ai.$ 
We have some $c>1/(2c_0)$ such that
\begin{equation}\label{const}
\ai<c\ri,~\Ri<c\ri,~
\ai<\Ri<c,~|\Ai|<c\Ri^d.
\end{equation}
for $\ep,$ and $i$ by \eqref{A0}, \eqref{A1} and \eqref{A3}.
We need to estimate $\Hi$ and $\grad\Hi$ on
$\Omega_{1i}^\ep
=\Bi\setminus \Ki,~\Omega_{2i}^\ep
=\{x\mid \ri/2<|x-x_i^\ep|-\ai<\ri\}$ and $\Omega_{3i}^\ep
=\{x\mid\Ri\leq|x-x_i^\ep|\}$ (see \Cref{cutoff}).
\begin{figure}
\begin{overpic}[width=7.4cm]{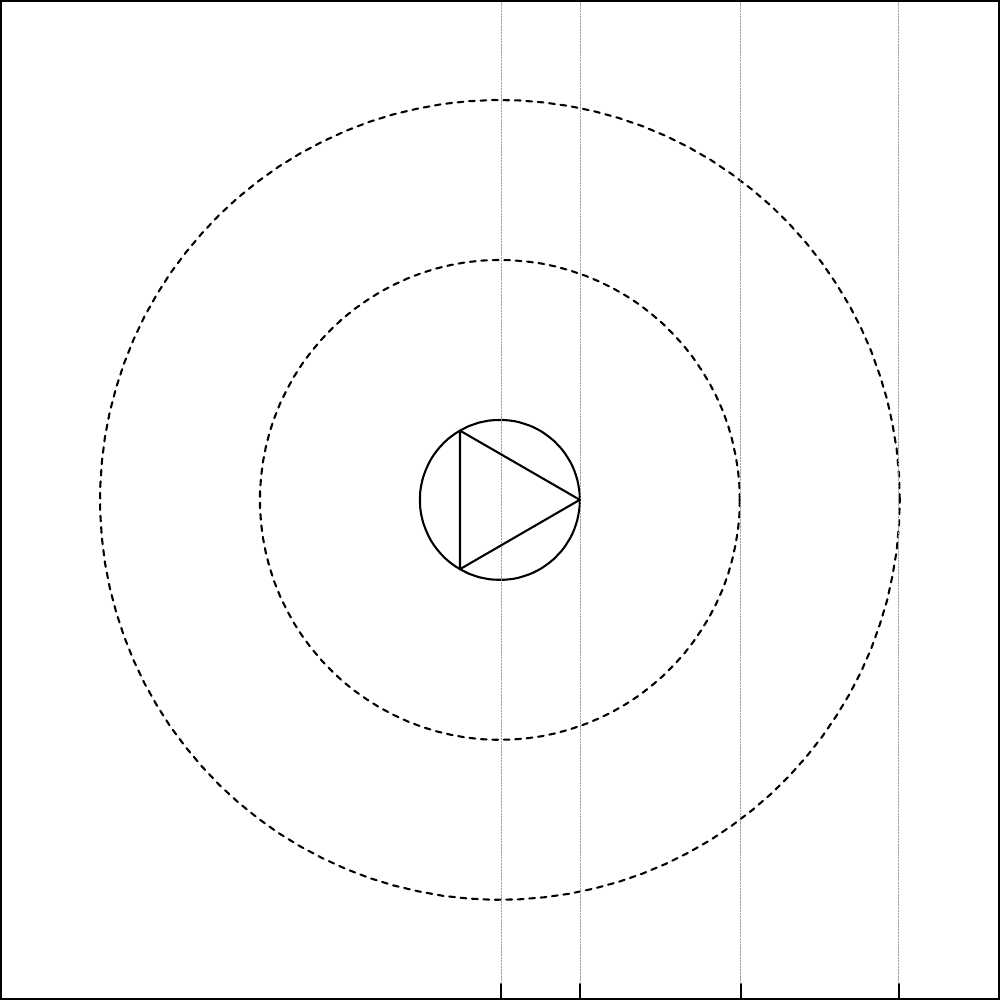}
\put(65,75){$\Omega_{2i}^\ep$}
\put(68,20){$\pd(\Omega_{1i}^\ep\cup\Ki)$}
\put(90,10){$\Omega_{3i}^\ep$}
\put(48,48){$K_i^\ep$}

\put(51,2){$\ai$}
\put(61,2){$\ri/2$}
\put(77,2){$\ri/2$}
\end{overpic}
\caption{Definition of $\Omega_{ji}^\ep~(j=1,2,3)$}\label{cutoff}
\end{figure}
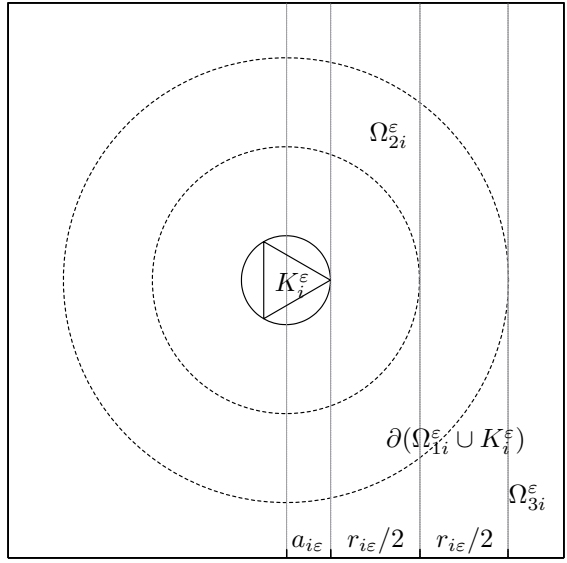
\begin{lem}
There exists $c'>c_0$ such that
\begin{align}
\label{pot2}\tag{P2}
&\norm{\pd^\alpha \Hi}_{L^2(\Omega_{1i}^\ep\cup \Ki)}^2
\leq c'\ai^{d-1-|\alpha|},\\
\label{pot3}\tag{P3}
&\norm{\pd^\alpha \Hi}_{L^2(\Omega_{2i}^\ep\cup\Omega_{3i}^\ep)}^2
\leq c'\ai^{2(d-2)}\Ri^{-d-2(-2+|\alpha|)}
\end{align}
for each $\ep,~i$ and $|\alpha|\leq 1.$
\begin{proof}
We have \begin{equation}\label{uniformdiam}\sup_i \ai\to 0\end{equation} since $\ai^{d-1}\leq\ai^{d-2}\Ri\leq\sum_i\ai^{d-2}\diam\Ai\to 0$ by \eqref{const}, \eqref{supertile} and \eqref{A4}.
\begin{equation}\label{pot2'}
\int_{1/2\leq|x-x_i^\ep|\leq c}|\pd^\alpha \Hi(x)|^2dx\leq c'\ai^{2(d-2)}
\end{equation} and \eqref{pot3}
follow from \eqref{pot1}, \eqref{uniformdiam} and \eqref{const}. Since \cite[2.2.2 Lemma 2.5.]{KM}, \eqref{Jung},  \eqref{uniformdiam} and $\lim_{t\to0}t\log t=0,$ we have
\[\norm{\pd^\alpha\Hi}_{L^2(B(x_i^\ep,3/4))}^2\leq c'\ai^d(\ai^{-2|\alpha|}+\delta_{d4}\log\ai+\ai^{-1}\delta_{d3})\leq(c'+1)\ai^{d-1-|\alpha|}.\]
\eqref{pot2} follows from it, \eqref{pot2'} and $\Omega_{1i}^\ep\cup\Ki\subset\Bi\subset B(x_i^\ep,c).$
\end{proof}
\end{lem}

%% file: p/30proofldc.tex
\begin{proof}[Proof of \Cref{ldcconvtop}]
Increasing property of capacity $\capa$ implies
\begin{equation}\label{capinc}
\capa(\Ki)\leq\capa(\cls{B(x_i^\ep,\ai)})=(d-2)S_d\ai^{d-2}.
\end{equation}
\eqref{A5}, \Cref{conv2} and $\sum_i\frac{\capa(\Ki)}{|\Ai|}1_{\Ai}
\leq(d-2)S_d\sum_i\frac{\ai^{d-2}}{|\Ai|}1_{\Ai}$ give \eqref{ldc3}.

Let $g\in\D(\Omega)$ and $b_\ep=\abs{\ev{\sum_i\frac{\capa(\Ki)}{|\Ai|}1_{\Ai}-\mu,g}}.$  Since \eqref{ldc}, \eqref{A4} and \eqref{capinc}, we have
\begin{align*}
&\abs{\ev{\sum_i\frac{\capa(\Ki)}{|\Ki|}1_{\Ki}-\mu,g}}
\leq\sum_i\capa(\Ki)\abs{\ev{\frac{1_{\Ki}}{|\Ki|}-\frac{1_{\Ai}}{|\Ai|},g}}+b_\ep\\
\leq&\sum_i\frac{(d-2)S_d\ai^{d-2}}{|\Ki||\Ai|}\int_{\Ai}\int_{\Ki}|g(x)-g(y)|dxdy+b_\ep\\
\leq&(d-2)S_d\Lip{g}\sum_i\ai^{d-2}\diam\Ai+b_\ep
\to0.
\end{align*}
Therefore, we have \eqref{ldc2}.
\end{proof}

%% file: p/31proof.tex
\subsection{Proof of the main result}\label{prf}
We show \Cref{resu} by applying \Cref{CM}.
Let $\phi\in C^2(\R,[0,1])$ satisfying $\phi(x)=\begin{cases}1&(x\leq1/2)\\0&(x>1)\end{cases}.$
Let $\phii(x)=\phi(\frac{|x-x_i^\ep|-\ai}{\ri})$ for each $\ep,~i$ (see \Cref{cutoff}). We remark
\[\supp\phii\subset\cls{\Omega_{1i}^\ep\cup\Ki},~\supp\grad\phii\subset\cls{\Omega_{2i}^\ep}\mbox{ (see \Cref{cutoff}) }.\]
We let
\[w^\ep=1-\sum_i\phii\Hi.\]
$\{w^\ep\}$ satisfies \eqref{H1} and \eqref{H2} by \eqref{cappot}.
\begin{proof}[Proof of \eqref{H3}]
It is enough to show $V^\ep\coloneqq\sum_i\phii\Hi\to 0$ weakly in $H^1(\Omega)$ to prove \eqref{H3}.
Since \eqref{A4}, \eqref{pot2}, \eqref{supertile} and \eqref{const}, we have
\begin{equation}\label{VepL2}
\norm{V^\ep}_{L^2(\Omega)}^2\leq \sum_i\norm{\Hi}_{L^2(\Omega_{1i}^\ep\cup\Ki)}^2\leq c'\sum_i\ai^{d-2}\diam\Ai\to 0.
\end{equation}
$\{V^\ep\}\subset H^1(\Omega)$ is bounded since \eqref{const}, \eqref{VepL2}, \eqref{A6}, \eqref{A2}, \eqref{pot2}, \eqref{pot3} and
\begin{align*}
\norm{\grad V^\ep}_{L^2(\Omega)}^2
&\leq2 \sum_i\qty(\norm{\grad\Hi}_{L^2(\Omega_
{1i}^\ep\cup\Ki
)}^2+\ri^{-2}\norm{\phi'}_\infty^2\norm{\Hi}_{L^2(\Omega_{2i}^2)}^2)\\
&\leq c''\sum_i\qty(\ai^{d-2}+\ai^{2(d-2)}\Ri^{-d+2}).
\end{align*}
The weak limit of any subsequence of $\{V^\ep\}$ converging weakly in $H^1(\Omega)$ is $0$ by Rellich's theorem and \eqref{VepL2}. Therefore, $V^\ep\to 0$ weakly in $H_0^1(\Omega).$
\end{proof}

%% file: p/32proofH5.tex
We denote $n^\ep$ the outer-word unit vector at $\bigcup_i\Bi.$ We see the sequence in \Cref{convdistr} appear in the proof of \eqref{H5} later.
\begin{lem}\label{convdistr}We have
\[-\sum_i\grad \Hi\cdot n^\ep\delta_{\pd \Bi}\to\mu\mbox{~in~}\D'(\Omega).\]
\begin{proof}
By \eqref{cappot}, we have
\begin{align}
\begin{aligned}
\label{capdecompose}
&\capa(\Ki)-\norm{\grad \Hi}_{L^2(\Omega_{3i}^\ep)}^2
=\norm{\grad(1-\Hi)}_{L^2(\Omega_{1i}^\ep)}^2\\
=&-\int_{\pd\Bi}(1-\Hi)\grad\Hi\cdot n^\ep dS
=-\int_{\pd\Bi}\grad\Hi\cdot n^\ep dS+b_i^\ep
\end{aligned}
\end{align}
where
$b_i^\ep=\int_{\pd\Bi}\Hi\grad\Hi\cdot n^\ep dS.$ Let $g\in\D(\Omega)$ and \[I_5^\ep=\ev{\sum_i\frac{\capa(\Ki)}{|\Ai|}1_{\Ai}-\mu,g}.\]
Since \eqref{capdecompose}, we have
\begin{align*}
&\ev{-\sum_i\grad\Hi\cdot n^\ep\delta_{\pd\Bi}-\mu,g}\\
&=\sum_i\ev{-\grad\Hi\cdot n^\ep\delta_{\pd\Bi}-\frac{\capa(\Ki)}{|\Ai|}1_{\Ai},g}+I_5^\ep\\
&=\sum_i\int_{\pd \Bi}\grad\Hi\cdot n^\ep(g(x_i^\ep)-g)dS+\sum_i\capa(\Ki)\qty(g(x_i^\ep)-\ev{\frac{1_{\Ai}}{|\Ai|},g})\\
&-\sum_ig(x_i^\ep)\norm{\grad\Hi}_{L^2(\Omega_{3i}^\ep)}^2+\sum_i g(x_i^\ep)b_i^\ep+I_5^\ep
\eqqcolon I_1^\ep+I_2^\ep+I_3^\ep+I_4^\ep+I_5^\ep.
\end{align*}
Since \eqref{supertile}, \eqref{const}, \eqref{pot1} and \eqref{A4}, we have
\[\abs{I_1^\ep}\leq c_0S_dc^{d-1}\Lip{g}\sum_i\ai^{d-2}\diam\Ai\to0\]
Since \eqref{capinc} and \eqref{A4}, we have
\[\abs{I_2^\ep}
\leq(d-2)S_d\Lip{g}\sum_i\ai^{d-2}\diam\Ai\to 0.\]
$I_3^\ep\to 0$ follows from \eqref{pot3} and \eqref{A2}. Since \eqref{pot1}, \eqref{const} and \eqref{A2}, we have
\[\abs{I_4^\ep}\leq S_dc_0^2\norm{g}_{L^\infty(\Omega)}\sum_i\ai^{2(d-2)}\Ri^{-d+2}\to0.\]
$I_5^\ep\to 0$ follows from \eqref{ldc}.
Therefore, the assertion follows.
\end{proof}
\end{lem}
We need to improve \Cref{convdistr} to convergence in $H^{-1}(\Omega).$ Then, we first show two lemmas below.
\begin{lem}\label{convnormpre1}
$\{\sum_i\ai^{d-2}\Ri^{-d}1_{\Bi}\}$ converges in the norm of $H^{-1}(\Omega).$
\begin{proof}
Let $\nu$ be the limit of \eqref{A5}, $g\in\D(\Omega),$ and $b_\ep=\abs{\ev{\nu-\sum_i\frac{\ai^{d-2}}{|\Ai|}1_{\Ai},g}}.$
Since \eqref{supertile} and \eqref{A4}, we have
\begin{align*}
\abs{\ev{\sum_i\frac{\ai^{d-2}}{|\Bi|}1_{\Bi}-\nu,g}}
&\leq\sum_i\ai^{d-2}\abs{\ev{\frac{1_{\Bi}}{|\Bi|}-\frac{1_{\Ai}}{|\Ai|},g}}+b_\ep\\
&\leq\Lip{g}\sum_i\ai^{d-2}\diam\Ai+b_\ep\to0.
\end{align*}
Since \eqref{const}, have 
\[\sum_i\frac{\ai^{d-2}}{|\Bi|}1_{\Bi}\leq \frac{c}{|B(0,1)|}\sum_i\frac{\ai^{d-2}}{|\Ai|}1_{\Ai}.\]
Applying \Cref{conv2} and \eqref{A5} for them, we have the assertion.
\end{proof}
\end{lem}
\begin{lem}\label{convnormpre2}
$\{\sum_i\ai^{d-2}\Ri^{-d+1}\delta_{\pd\Bi}\}$ converges in the norm of $H^{-1}(\Omega).$
\begin{proof}
Let
$\displaystyle q^\ep(x)=\sum_i\begin{dcases}
\frac{\ai^{d-2}}{2\Ri^d}(|x-x_i^\ep|^2-\Ri^2)&(x\in\Bi)\\
0&(x\notin\Bi)
\end{dcases}.$
Take any $h\in H_0^1(\Omega).$ We have $\norm{\grad q^\ep}_{L^2(\Omega)}^2\to0
$ by \eqref{A2}, and
\[\ev{\grad q^\ep,\grad h}
=\sum_i\ev{\ai^{d-2}\Ri^{-d+1}\delta_{\pd\Bi},h}-d\sum_i\ev{\ai^{d-2}\Ri^{-d}1_{\Bi},h}.\]
The assertion follows from them and \Cref{convnormpre1}.
\end{proof}
\end{lem}
Now we improve \Cref{convdistr} as below.
\begin{lem}\label{convnorm}
We have
\[\norm{-\sum_i\grad \Hi\cdot n^\ep\delta_{\pd\Bi}-\mu}_{H^{-1}(\Omega)}\to 0.\]
\begin{proof}
\eqref{pot1} and \eqref{const} gives $\sum_i|\grad\Hi\cdot n^\ep|\delta_{\pd\Bi}
\leq c''\sum_i\ai^{d-2}\Ri^{-d+1}\delta_{\pd\Bi}.$
Applying \Cref{conv2} for \Cref{convdistr} and \Cref{convnormpre2}, we have the assertion.
\end{proof}
\end{lem}
\begin{proof}[Proof of \eqref{H5}]
Consider $g,~\{v^\ep\},v$ satisfying the assumptions in \eqref{H5} and let $u^\ep=g v^\ep,~u=g v.$
We have
\begin{align*}
&\ev{\grad w^\ep,\grad u^\ep}=-\sum_i(\grad u^\ep,\grad(\Hi\phii))_{L^2(\Bi)}\\
=&-\sum_i(\grad u^\ep,\Hi\grad\phii)_{L^2(\Omega_{2i}^\ep)}-\sum_i(\grad u^\ep,(\phii-1)\grad \Hi)_{L^2(\Omega_{2i}^\ep)}\\
&-\sum_i(\grad u^\ep,\grad\Hi)_{L^2(\Bi)}\eqqcolon J_1^\ep+J_2^\ep+J_3^\ep.
\end{align*}
Since \eqref{cappot}, \Cref{convnorm} and $u^\ep\to u$ weakly in $H_0^1(\Omega),$ we have \[J_3^\ep=\ev{-\sum_i\grad\Hi\cdot n^\ep\delta_{\pd\Bi},u^\ep}\to\ev{\mu,u}.\] Since $\{u^\ep\}_{\ep}\subset H^1(\Omega)$ is bounded, \eqref{pot3}, \eqref{const} and \eqref{A2}, we have
\[\abs{J_1^\ep}
\leq \norm{\phi'}_\infty\sqrt{\sum_i\norm{u^\ep}_{H^1(\Omega_{2i}^\ep)}^2}\sqrt{\sum_i\ri^{-2}\norm{\Hi}_{L^2(\Omega_{2i}^\ep)}^2}
\leq c''\sqrt{\sum_i\ai^{2(d-2)}\Ri^{-d+2}}\to0\]
and $J_2^\ep\to 0$ similarly.
Therefore, \eqref{H5} follows.
\end{proof}
\begin{proof}[Proof of \Cref{resu}.]
\Cref{resu} follows from \Cref{CM}.
\end{proof}

%% file: p/40inv.tex
\section{Proof of \Cref{invperi}}\label{sinv}
In \Cref{prelinv}, we show \Cref{inv} which is a generalized result for \Cref{invperi}. However, \Cref{w-1infloc} for potential $\mu$ in \Cref{inv} is unusual.
In \Cref{potexamp}, we show $\mu\in L^d(\R^d)+L^\infty(\delta_S)$ satisfies \Cref{w-1infloc}, where $S$ is a piecewise $C^1$ surface.
\subsection{Inverse for homogenization problem}\label{prelinv}
\begin{ass}\label{tiles}
$A\subset\R^d$ is a bounded Borel set$,~\Lambda^\ep=\Lambda\subset\R^d~(\ep>0)$ is countable such that
\[\bigcup_{i\in\Lambda}(A+i)=\R^d\mbox{ and }\{A+i\}_{i\in\Lambda}\mbox{~is a family of disjoint sets}.\]
\end{ass}
We denote $\Ai=\ep(A+i)$ for each $\ep,~i$ (see \Cref{fignonperi}).
\Cref{invhint} below is a hint to construct $\{\Ki\}$ for given $\mu.$
\begin{lem}\label{invhint}
Let $\mu\in\D'(\R^d)$ be a non-negative Borel measure. Under \Cref{tiles}, we have
\[\sum_i\frac{\mu(\Ai)}{|\Ai|}1_{\Ai}\to\mu\mbox{ in }\D'(\R^d)
.\]
\begin{proof}
Let $\Omega_1=\cls{\bigcup_{x\in\Omega}B(x,\diam A)}.$ Let $g\in\D(\R^d).$ Since $\bigcup_i\Ai\subset\Omega_1$ for $\ep\ll1,$ we have
\begin{align*}
\abs{\ev{\sum_i\frac{\mu(\Ai)}{|\Ai|}1_{\Ai}-\mu,g}}
&=\abs{\sum_i\frac{1}{|\Ai|}\int_{\Ai}\int_{\Ai}\qty(g(x)-g(y))dxd\mu(y)}\\
&\leq\Lip{g}\sum_i\mu(\Ai)\diam\Ai
\leq\ep\Lip{g}\mu(\Omega_1)\diam A\to0.
\end{align*}
\end{proof}
\end{lem}
We consider a class of potentials as below.
\begin{ass}\label{w-1infloc}
$\mu\in W^{-1,\infty}_\loc(\R^d)^+$ and there exists $\{\nu^\ep\}\subset H^{-1}_\loc(\R^d)^+$ converging in $H^{-1}_\loc(\R^d)$ such that
\begin{equation}\label{inv1}\tag{INV}\sum_i\frac{\mu(\Ai)}{|\Ai|}1_{\Ai}\leq\nu^\ep~(\ep\ll1).
\end{equation}
\end{ass}
We use following estimate later.
\begin{lem}\label{inv2}
Let $\mu\in W^{-1,\infty}_\loc(\R^d)^+.$ and assume \Cref{tiles}. Then, there exists $c'>0$ such that
\[
\mu(\Ai)\leq c'\ep^{d-1}\mbox{ for each }\ep,~i\mbox{ satisfying }\Omega\cap \Ai\neq\emptyset.
\]
\begin{proof}
Take $g\in C^1(\R,[0,1])$ such that $g(x)=\begin{cases}
1&(|x|\leq 1)\\
0&(|x|>2)
\end{cases}$ and a bounded open set $V\supset\cls{\bigcup_{x\in\Omega}B(x,2\diam A)}.$ Let $X=W_0^{1,1}(V)$ with the norm  $\norm{\cdot}_X=\norm{\grad\cdot}_{L^1(V)}.$
Take some $y\in\Ai.$
We have
\begin{align*}\mu(\Ai)&\leq\int g\qty(\frac{|x-y|}{\diam\Ai})d\mu(x)
\leq \norm{\mu}_{X^*}\int \abs{\grad_xg\qty(\frac{|x-y|}{\diam\Ai})}dx\\
&\leq \norm{\mu}_{X^*}\norm{g'}_\infty|B(0,2)|(\ep\diam A)^{d-1}\end{align*}
for $\ep,~i$ satisfying $\Omega\cap \Ai\neq\emptyset.$
\end{proof}
\end{lem}
We expect $\{\Ki\}$ satisfying \eqref{invcapholes} below is a solution to the inverse homogenization problem by \Cref{invhint}.
\begin{ass}\label{invconstholes}
$\Ki=\emptyset$ or \Cref{assholec1} for each $\ep,i.$ Let $\cls{B(x_i^\ep,\ai)}$ be the minimal ball such that $\Ki\subset\cls{B(x_i^\ep,\ai)}.$
Assume there exist $c_1,~c_2>0$ such that $B(x_i^\ep,c_1\ep)\subset\Ai$ (see \Cref{fignonperi}) and \begin{equation}\label{invcapholes}
c_2\ai^{d-2}= \capa(\Ki)=\mu(\Ai)\end{equation}
for each $\ep,~i$.
\begin{figure}\centering
\begin{overpic}[width=8cm]{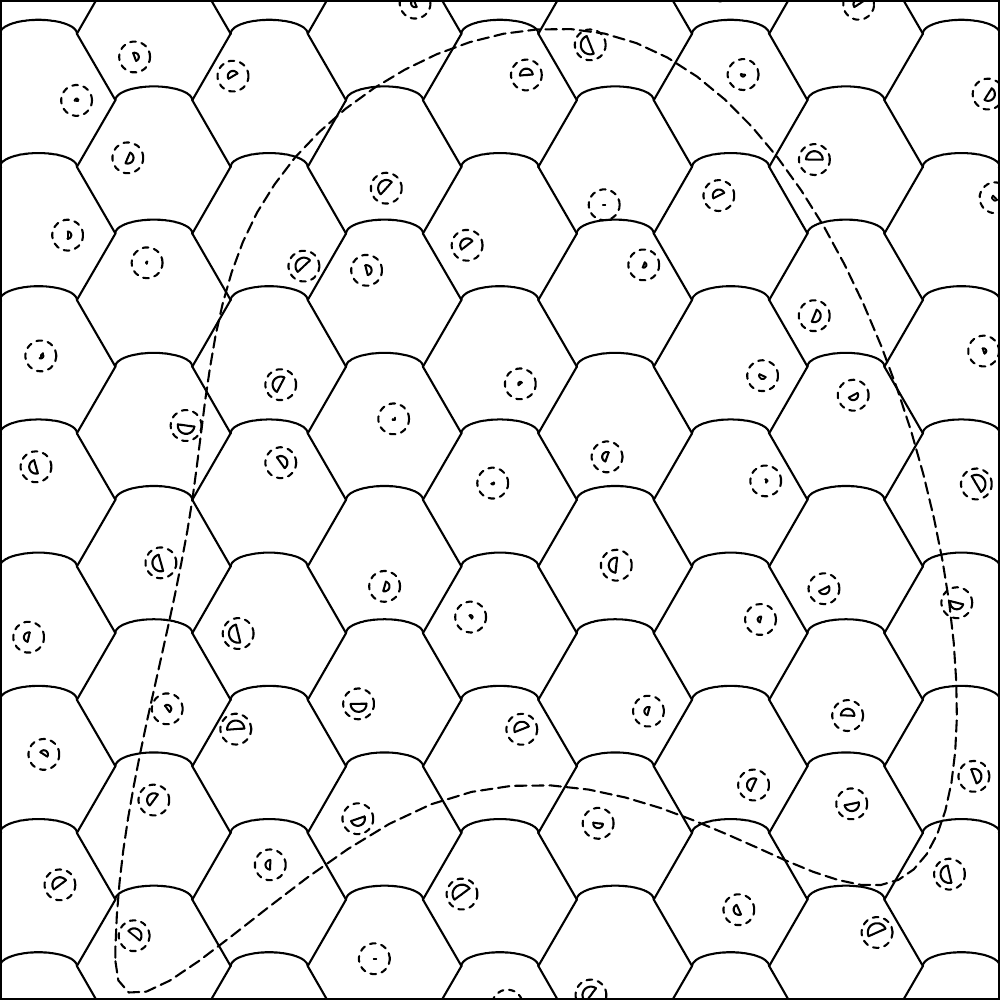}
\put(45,22){$\Omega$}
\put(59,30){$\Ai$}
\end{overpic}
\caption{A domain perforated by holes\label{fignonperi}}
\end{figure}
\begin{rem}
\eqref{invcapholes} is satisfied if $K$ is compact, $\capa(K)>0,~\Ki$ is congruence with $(\frac{\mu(\Ai)}{\capa(K)})^\frac{1}{d-2}K.$
\end{rem}
\end{ass}
Now we show \Cref{inv} which is a generalized result for \Cref{invperi}.
\begin{theo}\label{inv}
Under \Cref{w-1infloc,tiles,invconstholes},
the solutions $u^\ep$ and $u$ to \eqref{poipot} and \eqref{poi} satisfy $u^\ep\to u$ weakly in $H_0^1(\Omega).$
\end{theo}
\begin{proof}
We verify assumptions in \Cref{assalt} for $\Ri=c_1\ep.$ \eqref{A0} is clear.
\Cref{invconstholes} and \Cref{inv2} imply \eqref{supertile} and
\[c_2(\ai/\Ri)^{d-2}\leq c'c_1^{-d+2}\ep.\]
Therefore, \eqref{assalt1} follows. \eqref{assalt2} follows from \eqref{invcapholes}, $\diam\Ai=\ep\diam A.$
\eqref{ldc} and \eqref{assalt3} follow from \Cref{invhint}, \eqref{inv1} and \Cref{conv2}.
\end{proof}

%% file: p/41invpotpro.tex
\subsection{Potentials that the inverse homogenization problem is solved}\label{potexamp}
We first state some properties of potentials satisfying \Cref{w-1infloc}.
\begin{lem}\label{invsub}
Let $\nu$ be a positive Borel measure on $\R^d$, $\mu$ satisfy \Cref{w-1infloc}, $\nu\leq\mu.$ Then, $\nu$ also satisfies \Cref{w-1infloc} and \eqref{inv1}.
\begin{proof}
See \cite[Lemma 2.8.]{CM} to show $\nu\in W^{-1,\infty}_\loc(\R^d).$ \eqref{inv1} is clear.
\end{proof}
\end{lem}
\begin{lem}\label{invsum}
Let $\mu_1$ and $\mu_2$ satisfy \Cref{w-1infloc}. Then $\mu_1+\mu_2$ also satisfies it.
\end{lem}
\begin{ex}\label{invboudense}
Let $\mu$ satisfy \Cref{w-1infloc}, $f\in L^\infty(\mu,[0,\infty))$ and $d\mu'(x)=f(x)d\mu(x),$ Then $\mu'$ also satisfies them. 
\end{ex}

%% file: p/42invpotex.tex
Now we show $\mu\in L^d$ satisfies \Cref{w-1infloc}.
\begin{lem}\label{dencapLp}
Let $1\leq p<\infty,~\mu\in L^p_\loc(\R^d,[0,\infty)).$ Then, we have
\[\sum_i\frac{\mu(\Ai)}{|\Ai|}1_{\Ai}\to\mu\mbox{ in }L^p_\loc(\R^d).\]
\begin{proof}
Let $V$ be an open ball and $V_1=\bigcup_{x\in V}B(x,\diam A).$
Take $\{\mu_\delta\}_{\delta>0}\subset C^\infty(V_1)$ such that $\lim_{\delta\to0}\norm{\mu_\delta-\mu}_{L^p(V_1)}=0.$
We have \[\abs{\frac{\mu_\delta(\Ai)}{|\Ai|}-\mu_\delta(x)}
\leq\frac{\int_{\Ai}|\mu_\delta(y)-\mu_\delta(x)|dy}{|\Ai|}
\leq\ep\Lip{\mu_\delta}\diam A~(x\in\Ai).\]
H\"{o}lder's inequality gives \[\abs{\frac{\mu(\Ai)-\mu_\delta(\Ai)}{|\Ai|}}
\leq|\Ai|^{-1/p}\norm{\mu-\mu_\delta}_{L^p(\Ai)}\eqqcolon b_i^\ep\mbox{ for each }\ep>0,~i\in L.\]
Then, we have $\norm{\sum_i\frac{\mu(\Ai)-\mu_\delta(\Ai)}{|\Ai|}1_{\Ai}}_{L^p(V)}^p
\leq\sum_i\norm{b_i^\ep}_{L^p(\Ai)}^p
\leq\norm{\mu-\mu_\delta}_{L^p(V_1)}^p.$
Using these inequalities, we have
\[\norm{\sum_i\frac{\mu(\Ai)}{|\Ai|}1_{\Ai}-\mu}_{L^p(V)}
\leq2\norm{\mu-\mu_\delta}_{L^p(V_1)}+\ep\Lip{\mu_\delta}(\diam A)|V_1|^{1/p}\]
for each $\ep,~\delta>0,$ we have $\limsup_{\ep\to 0}\norm{\sum_i\frac{\mu(\Ai)}{|\Ai|}1_{\Ai}-\mu}_{L^p(V)}=0.$
\end{proof}
\end{lem}
\begin{ex}\label{invLd}
$\mu\in L^d_\loc(\R^d,[0,\infty))$ satisfies \Cref{w-1infloc}.
\begin{proof}
It follows from \eqref{embedding} and \Cref{dencapLp}.
\end{proof}
\end{ex}
Now we show surface measures satisfy \Cref{w-1infloc}. 
\begin{ex}\label{invgraph}
Let $S$ be a surface defined by the graph of $s\in C^1(\R^{d-1}).$ Then, $\mu=\delta_{S}$ satisfies \Cref{w-1infloc}.
\end{ex}
\begin{rem}
\Cref{invgraph} with $s=0$ and $A=(-1,1]^d$ corresponds to \cite[Example 2.9.]{CM}.
\end{rem}
\begin{proof}[Proof of \Cref{invgraph}.]
$\mu\in W^{-1,\infty}_\loc(\R^d)$ follows from trace theorem. Let
\[E_\ep={\{(x',x_d)\in\R^d\mid |x_d-s(x')|\leq\ep\}},~C=(1+\Lip{s})\diam A.\]
We have $\bigcup_{i,~\mu(\Ai)>0}\Ai\subset\bigcup_
{x\in S}B(x,\ep\diam A)\subset E_{C\ep}$ (see \Cref{surfest}).
\begin{figure}\centering
\begin{overpic}[width=8cm]{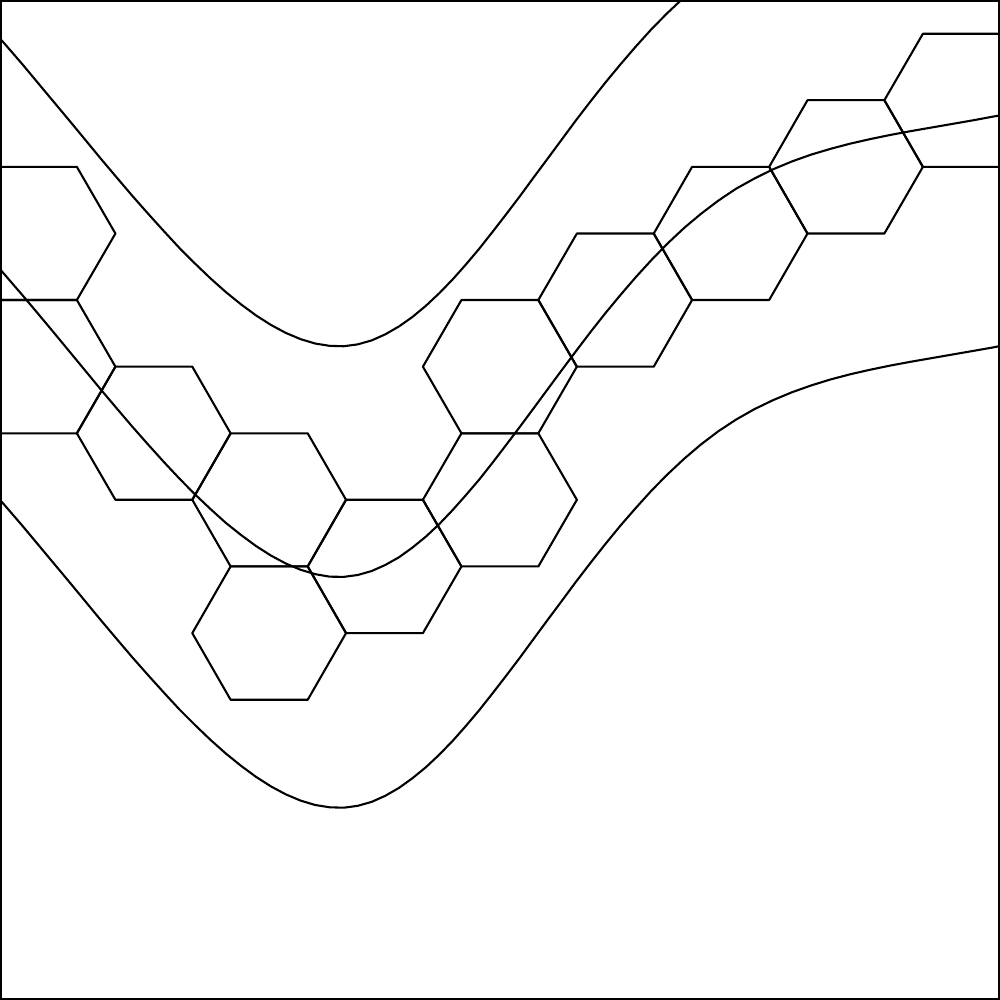}
\put(82,82){$\Ai$}
\put(40,53){$S\colon x_d=s(x')$}
\put(35,80){$x_d=s(x')+C\ep$}
\put(35,30){$x_d=s(x')-C\ep$}
\end{overpic}
\caption{Estimation for $\bigcup_{i~\mu(\Ai)>0}$\label{surfest} in \Cref{invgraph}}
\end{figure}
It and \Cref{inv2} imply
$\sum_i\frac{\mu(\Ai)}{|\Ai|}1_{\Ai}
\leq\frac{c'}{\ep|A|}1_{E_{C\ep}}.$
Let
\[\nu=\frac{1_{(0,\infty)}-1_{(-\infty,0)}}{2},~\nu_\ep(x)=\frac{1}{2\ep}\begin{cases}
-\ep&(x<-\ep)\\
x&(|x|\leq\ep)\\
\ep&(x>\ep)
\end{cases},\]
 $N(x',x_d)=\nu(x_d-s(x'))$ and $N_\ep(x',x_d)=\nu_\ep(x_d-s(x'))$ for each $\ep>0.$
We have $\pd_dN\in H^{-1}(\Omega).$ We can show $\pd_d N_\ep=(2\ep)^{-1}1_{E_\ep}$ in $\D'(\R^d)$ using $\pd\nu_\ep=(2\ep)^{-1}1_{[-\ep,\ep]}.$ For any $g\in\D(\Omega)$ with $\norm{g}_{H_0^1(\Omega)}\leq 1,$ we have
\[\abs{\ev{\frac{1_{E_{C\ep}}}{2C\ep}-\pd_d N,g}}
=\abs{\ev{N_{C\ep}-N,-\pd_dg}}
\leq c\norm{\nu_{C\ep}-\nu}_{L^2(\R)}\to 0\] for some $c>0.$
Therefore, we have \eqref{inv1}.
\end{proof}

\begin{ex}\label{invsurf}
Let $S$ be a piecewise $C^1$ surface on $\R^d.$ Then $\mu=\delta_S$ satisfies \Cref{w-1infloc}.
\begin{proof}
Sub-surface of $S$ in \Cref{invgraph} satisfies \Cref{w-1infloc} by \Cref{invsub}.
It and the implicit function theorem give measures $\mu_j~(j=1,...,N)$ such that $\mu=\sum_{j\leq N}\mu_j$ and each $\mu_j$ satisfies \Cref{w-1infloc}.  
$\mu$ satisfies \Cref{w-1infloc} by \Cref{invsum}.
\end{proof}
\end{ex}
Finally, we show \Cref{invperi} which is one of our main result.
\begin{proof}[Proof of \Cref{invperi}.]
$A=(-1,1]^d$ and $\Lambda=2\ep\Z^d$ satisfy \Cref{tiles}. \Cref{w-1infloc} follows from \Cref{invLd,invsum,invsurf,invboudense}. \Cref{invconstholes} follows from $\capa(\Ki)=(d-2)S_d\ai^{d-2}=\mu(\Ai)$ where $\ai=(\frac{\mu(\Ai)}{(d-2)S_d})^\frac{1}{d-2}.$
\end{proof}

%% file: p/51remain.tex
\section{Open problem}
The inverse for homogenization problem is not solved for every $\mu\in W^{-1,\infty}(\Omega)$ in this paper. It is expired to relax \eqref{inv1} if possible.

This paper requires uniform convergence for $\diam\Ki$ as \eqref{uniformdiam}. We want to study about the case for which the volumes of holes tend to $0$ (the diameters may not tend to $0$).

%% file: p/99-refs.tex
\nocite{CICE}
\bibliographystyle{plain}
\bibliography{Cite/CICE,Cite/RT,Cite/CM,Cite/Con,Cite/KM,Cite/NSL,Cite/Hru2,Cite/RATE}